\documentclass[11pt]{article}      

\usepackage{amssymb}      
\usepackage{amsmath}      
\usepackage{amsthm}      
\usepackage{amsbsy}      
\usepackage{amscd}      
\usepackage[dvips]{graphicx}

\theoremstyle{plain}      
\newtheorem{theorem}{Theorem}[section]      
\newtheorem{lemma}{Lemma}[section]      
\newtheorem{corollary}[theorem]{Corollary}      
\newtheorem{proposition}{Proposition}[section]

\newtheorem{definition}{Definition}[section]          
\theoremstyle{remark}      
\newtheorem{remark}{Remark}[section]

\newcommand{\Q}{{\mathbb{Q}}}        
\newcommand{\Z}{{\mathbb{Z}}}   
   
\newcommand{\C}{{\mathbb{C}}}      
\newcommand{\R}{{\mathbb{R}}}      

\newcommand{\ro}{{\widetilde{\rho}}}

\begin{document}

\date{\today}

\title{Zariski density and finite quotients of mapping class groups}         
\author{ Louis Funar\\      
\small \em Institut Fourier BP 74, UMR 5582 \\      
\small \em University of Grenoble I \\      
\small \em 38402 Saint-Martin-d'H\`eres cedex, France  \\      
\small \em e-mail: {\rm funar@fourier.ujf-grenoble.fr}  
}

\maketitle 

\begin{abstract} 
Our main result is that the image
of  the quantum representation of a central extension of the mapping 
class group of the genus $g\geq 3$ closed orientable surface 
at a prime $p\geq 5$ is a Zariski dense 
discrete subgroup of some higher rank algebraic 
semi-simple Lie group $\mathbb G_p$ defined over $\Q$. 
As an application we find that, for any prime $p\geq 5$ 
a central extension of the genus $g$ 
mapping class group surjects onto  
the finite  groups $\mathbb G_p(\Z/q\Z)$, 
for all but finitely many primes $q$. 
This method provides infinitely many 
finite quotients of a given mapping class group  
outside the realm of symplectic groups. \\

 
\end{abstract}

\maketitle

\section{Introduction and statements}

The  aim of the this paper is to obtain a largeness result for the 
images of quantum representations of mapping class groups in genus at least 
$3$. The main motivation is the construction of large families 
of finite quotients of (central extensions of the) mapping class groups 
by using the strong approximation theorem. 
This method furnishes a large supply of  
finite quotients of  mapping class groups   
outside the realm of symplectic groups. Similar results were obtained 
independently by Masbaum and Reid in \cite{MRe}.  Earlier, Looijenga 
has proved in \cite{Loo} that the images of Prym representations  
(associated to finite abelian groups) of  suitable 
finite index subgroups of mapping class groups  
and of subgroups from the Johnson filtration are arithmetic groups.

Some largeness results in this direction are already known. 
In \cite{F} we showed that the images are infinite and 
non-abelian (for all but finitely many explicit cases) 
using earlier results of Jones who proved in 
\cite{Jones} that the same holds true for the braid group 
representations factorizing through the 
Temperley-Lieb algebra at roots of unity. Masbaum then found in \cite{Mas98} 
explicit elements of infinite order in the image.   
General arguments concerning Lie groups actually show 
that the image should contain a free 
non-abelian group. Furthermore, Larsen and Wang proved (see \cite{LW}) 
that the image  of the quantum representations of the mapping class groups at 
roots of unity of the form 
$\pm \exp\left(\frac{2(p+1)\pi i}{4p}\right)$, for prime $p\geq 5$,  
is dense in the projective unitary group.

In a previous paper \cite{FK11} the authors proved that the 
images are large in the sense that they contain explicit 
free non-abelian groups. Now, each image is contained into some unitary group 
of matrices with cyclotomic integers entries, by \cite{GM}. 
The latter group can be embedded as an irreducible higher rank lattice 
in a semi-simple Lie group $\mathbb G_p(\R)$ (depending on the genus and 
the order of the roots of unity)  obtained by restriction of scalars. 
The main result of this paper  strengthens the largeness property above 
by showing that, in general, the image of a  
quantum representation is Zariski dense in  the non-compact group 
$\mathbb G_p(\R)$.

Let us introduce now some terminology.  Recall that in \cite{BHMV} 
the authors defined the TQFT functor $\mathcal V_{p}$, for every $p\geq 3$  
and a primitive root of unity $A$ of order $2p$.
These TQFT should correspond to the so-called 
$SU(2)$-TQFT, for even $p$ and to 
the $SO(3)$-TQFT, for odd $p$ (see also \cite{LW} for another 
$SO(3)$-TQFT). 

\begin{definition}\label{qrep}
Let $p\in\Z_+$, $p\geq 3$ and $A$ be a primitive 
$2p$-th root of unity. 
The  quantum representation $\rho_{p,A}$ 
is the projective representation of  the mapping class group 
associated to the TQFT $\mathcal V_{p}$ at the root of unity  $A$. 
We denote therefore 
by $\ro_{p,A}$ the linear representation of the central extension 
$\widetilde{M_g}$ of the mapping class groups $M_g$ (of the genus $g$ closed orientable surface) which resolves the 
projective ambiguity of $\rho_{p,A}$ (see \cite{Ger,MR}).  
Furthermore, $N(p,g)$ will denote the dimension of the space of conformal 
blocks associated by the TQFT $\mathcal V_{p}$ to the closed 
orientable surface of genus $g$. 
\end{definition}

\begin{remark}
The unitary TQFTs arising usually correspond to the following 
choices of the root of unity: 
\[ A_p=\left\{\begin{array}{ll}
-\exp\left(\frac{2\pi i}{2p}\right), & {\rm if}\: p\equiv 0({\rm mod}\: 2);\\
-\exp\left(\frac{(p+1)\pi i}{p}\right) , & {\rm if}\: p\equiv 1({\rm mod}\: 2).\\
\end{array}\right. \]
\end{remark}

For  prime $p\geq 5$  we denote 
by ${\mathcal O}_p$ the  ring of cyclotomic integers 
${\mathcal O}_p=\Z[\zeta_p]$, if   
$p\equiv -1({\rm mod}\: 4)$ and ${\mathcal O}_p=\Z[\zeta_{4p}]$, if  
$p\equiv 1({\rm mod}\:4)$ respectively, where $\zeta_p$ is a primitive
$p$-th root of unity.
The main result of \cite{GM} states that, for every 
prime  $p\geq 5$, there exists a free ${\mathcal O}_p$\,-lattice 
$S_{g,p}$ in the $\C$-vector space of conformal 
blocks associated by  the TQFT ${\mathcal V}_p$ to the genus $g$
closed orientable surface and a non-degenerate Hermitian  
${\mathcal O}_p$-valued form on 
$S_{g,p}$ such that  (a central extension of) the mapping class group 
preserves $S_{g,p}$  and keeps invariant the Hermitian form. 
Therefore the image of the mapping class group consists of 
unitary matrices (with respect to the Hermitian form) with 
entries in ${\mathcal O}_p$. Let  $P\mathbb U({\mathcal O}_p)$ be the group 
of all such matrices, up to scalar multiplication. 
For the sake of simplicity of the exposition we will consider only primes 
$p$ satisfying $p\equiv -1({\rm mod}\; 4)$, from now on. Similar results 
hold for the remaining primes with only minor modifications.

It is known that $P\mathbb U({\mathcal O}_p)$ is an irreducible 
lattice in a semi-simple Lie group $\mathbb P\mathbb G_p(\R)$ obtained 
by the so-called restriction of scalars construction from the 
totally real cyclotomic field $\Q(\zeta_p+\zeta_p^{-1})$ to $\Q$. 
Specifically, let us denote by $\mathbb G_p(\R)$ the product 
$\prod_{\sigma\in S(p)}S\mathbb U^{\sigma}$. Here $S(p)$ stands for 
a set of representatives for the classes of complex 
valuations $\sigma$ of $\mathcal O_p$ 
modulo complex conjugacy. 
The factor $S\mathbb U^{\sigma}$ is the 
special unitary group associated to the 
Hermitian form conjugated by $\sigma$, thus corresponding to some 
Galois conjugate root of unity.  Denote also 
by $\ro_p$ and $\rho_p$ the representations  
$\prod_{\sigma\in S(p)} \ro_{p,\sigma(A_p)}$ and 
$\prod_{\sigma\in S(p)} \rho_{p,\sigma(A_p)}$, respectively. 
In \cite{DW}  the authors proved that the restriction of 
$\ro_p$ to the universal central extension  $\widetilde{M_g}^{\rm univ}$ 
of $M_g$  -- which is a subgroup of $\widetilde{M_g}$ of index 12 -- takes values 
in $S\mathbb U$. This implies that $\ro_p(\widetilde{M_g})\subset S\mathbb U$ 
for $g\geq 3$ and  prime $p\geq 5$.   

\vspace{0.2cm}\noindent
Notice that the $\mathbb G_p(\R)$ is the set of real points  
of a semi-simple algebraic group $\mathbb G_p$ 
defined over $\Q$.

\vspace{0.2cm}\noindent
Our main result can be stated now as follows: 

\begin{theorem}\label{ZD}
Suppose that $g \geq 3$ and  $p\geq 5$ is prime.  
Then $\ro_p(\widetilde{M_g})$ is a discrete Zariski dense subgroup 
of $\mathbb G_p(\R)$ whose projections onto the simple factors of 
$\mathbb G_p(\R)$ are topologically dense. 
\end{theorem}

\begin{remark}
A similar result holds for the $SU(2)$-TQFT. Specifically let  
$p=2r$ where $r\geq 5$ is prime. 
According to (\cite{BHMV}, 1.5) there is an isomorphism 
of TQFTs between $\mathcal V_{2r}$ and $\mathcal V_2'\otimes \mathcal V_r$, 
and hence the projection on the second factor gives us 
a homomorphism $\pi:\ro_{2r}(\widetilde{M_g})\to \mathbb G_r(\R)$. 
Furthermore, the image of the TQFT representation associated to $\mathcal V_2'$ 
is finite. Therefore, $\pi\circ \ro_{2r}(\widetilde{M_g})$ is a 
discrete Zariski dense subgroup of $\mathbb G_r(\R)$. 
Notice that the result holds also for $g=2$ and prime $p\geq 7$ 
using the modifications from \cite{FK11} in the constructions of 
free non-abelian subgroups in the image. We skip the details. 
\end{remark}

We now consider the Johnson filtration by the 
subgroups $I_g(k)$ of the mapping class group $M_g$ of the closed orientable 
surface of genus $g$, consisting of those 
elements having a trivial outer action on the 
$k$-th nilpotent quotient of the fundamental 
group of the surface, for  some $k\in\Z_+$.

The main application of our density result is the following: 
\begin{theorem}\label{fquot}
For every $g\geq 3$ and prime $p\geq 5$,    
there exists some homomorphism 
$\widetilde{M_g}\to \mathbb G_p(\Z/q^k\Z)$, whose restriction to 
$I_g(3)$ is surjective for all 
large enough primes $q$ and all $k\geq 1$. In particular, the surjectivity holds 
also for $\widetilde{M_g}$, the Torelli group $I_g(1)$,  
the Johnson kernel $I_g(2)$, or any finite index subgroup 
of $\widetilde{M_g}$, respectively. 
\end{theorem}
\begin{proof}
The pull-back of the central extension $\widetilde{M_g}$ on the Torelli subgroup 
$I_g(1)$ is trivial, because the generator of $H^2(M_g)$ is the pull-back 
of the generator of $H^2(Sp(2g,\Z))$, when $g\geq 3$. Thus 
$I_g(3)$ embeds into $\widetilde{M_g}$.  The image of the central factor 
is finite and $\rho_p(I_g(3))$ is of finite index in $\rho_p(M_g)$, 
according to proposition \ref{chain3}. Therefore,  $\ro_p(I_g(3))$ is of finite index 
in  $\ro_p(\widetilde{M_g})$ and hence Zariski dense in $\mathbb G_p(\R)$. 
Furthermore,  we will use 
the  following version of the 
strong approximation theorem due to Nori  
(see (\cite{No}, Thm.5.4) and also \cite{Weis}):
Let $G$ be a connected  linear algebraic group $G$ defined over $\Q$  
and $\Lambda\subset G(\Z)$ be a Zariski dense subgroup. 
Assume that $G(\C)$ is simply connected. Then the 
completion of $\Lambda$ with respect to the congruence  
topology induced from $G(\Z)$ is an open subgroup in 
the group $G(\widehat{\Z})$ of points of $G$ 
over the pro-finite completion $\widehat{\Z}$ of $\Z$.   
We now consider the group $G=\mathbb G_p$ which satisfies 
the assumptions of Nori's theorem.  
If we take $\Lambda$ to be a finite index subgroup of   
$\ro_p(\widetilde{M_g})$ the strong approximation theorem implies 
our claim for $k=1$. Then a classical result due to 
Serre  (see \cite{Serre}) for  $GL(2)$ 
and extended by Vasiu (see \cite{Vas}) to all reductive 
linear algebraic groups defined over $\Q$ improves the 
surjectivity statement to all $k\geq 1$.  
\end{proof}

\begin{remark}
The  arithmetic group $\mathbb G_p(\Z)$, for $g\geq 3$ and  prime 
$p\geq 5$ has the congruence property.  
This follows from results of Tomanov (see \cite{Tom}, Main Thm. (a)) 
and Prasad and Rapinchuk (see \cite{PR}, Thm. 2.(1) and Thm. 3)
on the congruence kernel for $\Q$-anisotropic algebraic groups of 
type $^2A_{n-1}$, with $n\geq 4$. In this respect it would be interesting 
to construct finite quotients of 
the residually finite (see \cite{F2})  group $\widetilde{M_g}$,   
other than the quotients of groups of the form $\mathbb G_p(\Z/q\Z)$. 
Moreover, $\mathbb G_p(\Z)$ 
is cocompact in $\mathbb G_p(\R)$, since 
it is $\Q$-anisotropic, by a classical result of Borel and Harish-Chandra 
(see \cite{BH}).    
\end{remark}

\begin{corollary}\label{fquot2}
For any prime $p\geq 5$ and $g\geq 3$ 
there exists a homomorphism  $M_g\to \mathbb P\mathbb G_p(\Z/q^k\Z)$,  
whose restriction to a given finite index subgroup of $M_g$ 
is surjective for all large enough primes $q\geq 5$ and all $k\geq 1$.
Here $\mathbb P\mathbb G_p(\R)$ is the product of the projective 
unitary groups whose associated  special unitary groups 
occur as factors of $\mathbb G_p(\R)$.  
\end{corollary}
\begin{proof}
The image of the center of $\widetilde{M_g}$ by the homomorphism 
$\ro_p$ is contained in the 
centralizer of $\ro_p(\widetilde{M_g})$. The Zariski density 
result above implies that the centralizer is contained in the 
product of the centers of each simple factor of $\mathbb G_p(\R)$.  
This proves the claim. 
\end{proof}

\begin{remark}
\begin{enumerate}
\item  The first construction of finite quotients of mapping 
class group by this method was given by Masbaum in \cite{Mas08}.
\item The set of finite quotients of a particular $M_g$ 
(with $g\geq 3$) provided by Theorem \ref{fquot} is rather large.
Indeed,  $\mathbb P\mathbb G_p(\Z/q\Z)$  
are  finite groups of Lie type of arbitrarily large rank.  
In particular, the alternate group on $m$ elements  
is contained into some  $\mathbb P\mathbb G_p(\Z/q\Z)$, for large enough 
$p$ and $q$, and hence into some finite quotient of $M_g$. Therefore, 
every finite group embeds in some finite quotient of the 
genus $g\geq 3$ mapping class group. 
This answers a question of U. Hamenstaedt, 
first settled by Masbaum and Reid in \cite{MRe}.    
\item In \cite{FP} we already obtained results showing  that 
a given mapping class group 
has many more finite quotients than the family of all 
symplectic groups, as it can be measured by  the torsion of 
their  (essential) $2$-homology groups. 
\item The number $n(p)$ of the non-compact factors in $\mathbb G_p(\R)$ 
goes to infinity with $p$.  
\end{enumerate}
\end{remark}

\vspace{0.2cm}\noindent
Let $QH(G)$ denote the vector space of   
quasi-homomorphisms of the group $G$, namely of maps $\varphi:G\to \R$ 
for which $\sup_{a,b\in G}|\partial \varphi(a,b)|< \infty$, where 
$\partial \varphi(a,b)=\varphi(ab)-\varphi(a)-\varphi(b)$ is the 
boundary 2-cocyle. The quasi-homomorphism $\varphi$ is homogeneous 
if $\varphi(a^n)=n\varphi(a)$, for every $a\in G$ and $n\in\Z$. 
We consider the quotient $\widetilde{QH}(G)$ 
of $QH(G)$ by  the submodule generated by  the bounded 
quasi-homomorphisms of $G$ and the group homomorphisms.  
It is well-known (see \cite{BI}) 
that $\widetilde{QH}(G)$ is isomorphic to 
the kernel of the comparison homomorphism 
$H^2_b(G,\R)\to  H^2(G,\R)$, where $H^2_b(G,\R)$ denotes the 
second bounded cohomology group of $G$.  

\vspace{0.2cm}\noindent
If $\mathcal X$ is an irreducible Hermitian symmetric space of non-compact 
type and $I(\mathcal X)$ its isometry group, then denote by  
$\mathcal K_{I(\mathcal X)}$ the generator of the continuous bounded 
second cohomology  group $H^2_{cb}(I(\mathcal X),\R)$  of $I(\mathcal X)$. 
This class is defined, for instance, by the Dupont cocycle
$c_{I(\mathcal X)}:I(\mathcal X)\times I(\mathcal X)\to \Z$, given by:   
\[ c_{I(\mathcal X)}(g_1,g_2)=\frac{1}{4\pi}\int_{\Delta(g_1(x_0),g_2(x_0),g_1g_2(x_0))}\omega,  
\; {\rm }\; \;\;\; g_1,g_2\in I(\mathcal X),\]
where $\omega$ is the K\"ahler form on $\mathcal X$, $x_0\in \mathcal X$ 
and $\Delta(x,y,z)$ denotes an oriented smooth triangle on $\mathcal X$ 
with geodesic sides. 
Although the interior of the triangle with geodesic sides is not uniquely 
defined the value of the cocycle is well-defined because $\omega$ is closed 
(see \cite{BI} for details).
It is also known that the class of $c_{I(\mathcal X)}$ in 
the continuous second cohomology group $H^2_c(I(\mathcal X),\R)\cong\R$ 
is a generator. 
Bestvina and Fujiwara proved in \cite{BF} that $\widetilde{QH}(M_g)$ 
is infinitely generated. Quantum representations 
yield an explicit family of quasi-homomorphisms on the mapping class groups 
generalizing the Rademacher function on $SL(2,\Z)$ (which can be obtained 
for the one-holed torus), as follows:   
\begin{corollary}\label{quasi}
Let $g\geq 3$, $p\geq 5$ be  a prime number and 
$SU(m,n)$ be a  
non-compact simple factor of $\mathbb G_{p}$ corresponding to the 
primitive root of unity $A$. Let 
 $\widetilde{M}_g^{\rm univ}\subset \widetilde{M_g}$ be the 
universal central extension of $M_g$. Then there exists a unique 
quasi-homomorphism 
$L_{A}:\widetilde{M}_g^{\rm univ}\to \R$ verifying  
\[ \partial L_{A}= \frac{1}{{4\pi}(m+n)}
\widetilde{\rho}_{p,A}^*\, c_{SU(m,n)}
\]
Furthermore, the classes of those $L_{A}$, for which $1\leq m <n$,  
are linearly independent over $\Q$ in $\widetilde{QH}(\widetilde{M_g}^{\rm univ})$. 
\end{corollary}
\begin{proof}
Burger and Iozzi proved in (\cite{BI}, Thm 1.3) that for each set of  
pairwise non-equivalent (i.e. not conjugate by an 
isometry of the corresponding symmetric spaces) Zariski dense 
representations $\rho_j:\Gamma \to SU(m_j,n_j)$,  
of a finitely generated group $\Gamma$, for which 
$1\leq m_j<n_j$, the elements 
$\rho_j^*(\mathcal K_{SU(m_j,n_j)})\in H^2_b(\Gamma,\R)$ 
are linearly independent over $\Z$. 
On the other hand  $H^2(\widetilde{M}_g^{\rm univ})=0$ because 
$\widetilde{M}_g^{\rm univ}$ is a universal central extension and hence it does 
not have any non-trivial central extension. 
Therefore, the cocycle $\widetilde{\rho}_{p,A}^*c_{SU(m,n)}$ 
corresponds to an element of $\widetilde{QH}(M_g)$ and the 
corollary follows from  the above cited result in \cite{BI}. 
\end{proof}
\begin{remark}
 If $\rho': \widetilde{M_g}^{\rm univ}\to SU(m,n)$ 
is some Zariski dense representation  and $L'$ is the corresponding 
quasi-homomorphism, then $L_{A}=L'$ in $\widetilde{QH}(M_g)$ 
only if $\rho'$ is conjugate to $\ro_{p,A}$ (see \cite{BI}). 
Moreover, let 
$\overline{L}_{A}$ denote the unique homogeneous 
quasi-homomorphism in the class of $L_{A}$. 
Then $\overline{L}_{A}$ is a class function  
(i.e. invariant on conjugacy classes) on $\widetilde{M_g}^{\rm univ}$ 
which encodes all information about the representation $\ro_{p,A}$.
\end{remark}
\begin{remark}
If the groups $\rho_p(M_g)$ were finitely presented and the 
number of relations in some group presentation 
were uniformly bounded (for fixed $g$),  
then the rank of $H^2(\rho_p(M_g))$ 
would also be uniformly bounded  and  
our density theorem  would imply that 
$\widetilde{QH}(\rho_{p}(M_g))$  cannot be trivial for large enough $p$.  
For instance the uniform bound holds if we replace $\ker \rho_p$ above  
by the subgroup generated by $p$-th powers of all  
Dehn twists and one might conjecture that the two subgroups 
coincide in genus $g\geq 3$. Eventually, if $\widetilde{QH}(\rho_{p}(M_g))$ 
were infinite dimensional then $\rho_p(M_g)$ would not be boundedly generated. 
\end{remark}

Recall now the following recent result due to Salehi Golsefidy and Varj\'u 
(see \cite{SGV}, Thm. 1 and Cor. 5): Let $\Gamma\subset GL(N,\Z)$  be a group 
generated by a symmetric set $S$ and denote by $\pi_q:GL(N,\Z)\to GL(N,\Z/q\Z)$ 
the reduction mod $q$. Then the family of Cayley graphs of the 
groups $\pi_q(\Gamma)$ with generator systems $\pi_q(S)$ forms 
a family of expanders (see \cite{LZ} for details about expanders), 
where $q$ runs through square-free integers, 
if and only if the connected component of the Zariski closure of 
$\Gamma$ is perfect. This family of Cayley graphs forms expanders  iff  
$\Gamma$ has property  $\tau$ with respect to the family of finite 
quotients  $\pi_q(\Gamma)$, where $q$ runs through square-free integers, 
namely if there exists some $\varepsilon>0$  (depending on $S$) such that for every 
unitary representation $\rho$ of $\Gamma$ into some Hilbert space $H$ 
factorizing through some $\pi_q(\Gamma)$ without invariant nontrivial 
vector and every unit vector $v\in H$ we have $|\rho(s)v-v| \geq \varepsilon$, 
for some $s\in S$. Here $|*|$ denotes the norm in $H$. 
Observe now that the (complex) Zariski closure of 
$SU(m,n)$ in $GL(m+n,\C)$ is $SL(m+n,\C)$.     
In particular, our density result implies that 
the Zariski closure of $\ro_p(\widetilde{M_g})$ (within the appropiate 
product of copies of $GL(N(p,g),\C)$) is the complex Zariski closure of $\mathbb G_p$, 
namely the product of several copies of $SL(N(p,g),\C)$. 
Since the Zariski closure is perfect then  by (\cite{SGV}, Thm.1, but Cor.5 would also suffice) we obtain the following:  

\begin{corollary}\label{tau}
For every prime $p\geq 5$,  the groups 
$\ro_p(\widetilde{M_g})$ (and  
$\ro_p(I_g(k))$, $k\leq 3$) have property $\tau$ with respect to 
the family of finite quotients induced by $\pi_q$, 
where $q$ runs through square-free integers.   
\end{corollary}

\begin{remark}
Bourgain and Gamburd proved (see \cite{BGa,BGa2}) that a dense subgroup 
of $SU(N)$ generated by  matrices with algebraic entries, and 
in particular $\ro_p(\widetilde{M_g})$, has a spectral gap 
with respect to the natural action on $L^2(SU(N))$.  
\end{remark}
\begin{remark}
On the other hand these results 
cannot be considered as evidence in the favor of the claim that  
$\ro_p(\widetilde{M_g})$ has  property $\tau$ with respect 
to the family of {\em all } its finite quotients. In fact, by \cite{BG} 
there exists a free subgroup on two generators 
$L\subset \ro_p(\widetilde{M_g})$ which is Zariski 
dense in $\mathbb G_p(\R)$. Although a free group has not property $\tau$ 
with respect to all its finite quotients, 
the group  $L$ has property $\tau$ with respect to the family of finite 
quotients induced by $\pi_q$, where $q$ runs through square-free integers, 
by the result from \cite{SGV}.  It would be very 
interesting to find whether $\widetilde{M_g}$ has property $\tau$ 
with respect to the family of finite quotients  
$\pi_{q}(\ro_p(\widetilde{M_g}))$, where $q\geq q(p)$ are sufficiently large, 
possibly square-free and $p$ runs over the primes. 
\end{remark}

\vspace{0.2cm}\noindent 
{\bf Acknowledgements.}  I am indebted to Yves Benoist, Philippe Gille, 
Olivier Guichard, 
Greg Kuperberg, Toshitake Kohno, Greg McShane, 
Gregor Masbaum, Mahan Mj, Alan Reid, Adrian Vasiu and Tyakal N. Venkataramana   
for useful discussions and advice. 
The author was partially supported by the ANR grant ModGroup.

\section{Proof of Theorem \ref{ZD}}

The starting point is the following result from \cite{LW}, 
subsequently improved in \cite{DW}: 

\begin{proposition}[\cite{LW}]\label{LW}
If $g\geq 2$, $p\geq 5$ is prime and $(g,p)\neq (2,5)$, then 
$\rho_{p,A_p}(M_g)$ is (topologically) dense 
in $SU(N(p,g))$, where 
$N(p,g)$ is the dimension of the space of conformal blocks 
associated to the closed orientable 
surface $\Sigma_g$ of genus $g$. 
\end{proposition}

A cautionary remark is in order. There are two TQFTs 
which might reasonably be called the $SO(3)$-TQFT and although their 
differences are rather minimal, they are distinct (see the discussion in 
\cite{LW}). We work here with the version from 
\cite{BHMV}, in order to have a common approach for both 
$SU(2)$-TQFT and the $SO(3)$-TQFT.  However, the proof from 
\cite{LW}, although stated for one version,  works also for the 
other version of the TQFT.

The Hermitian form $H$ associated to the TQFT $\mathcal V_p$ and the 
root of unity $A$ has signature $(N_+(p,g,A), N_-(p,g,A)$, where 
$N_+(p,g,A)+ N_-(p,g,A)= N(p,g)$. We prove first: 

\begin{proposition}\label{weakZD}
Let $g\geq 3$ and $A$ be any primitive $2p$-th root of unity. 
Then $\rho_{p,A}(M_g)$ is (topologically) dense 
in $SU(N_+(p,g,A), N_-(p,g,A))$. 
\end{proposition}

Two representations into the same group $G$ are called 
equivalent if there exists an isomorphism of $G$ 
which intertwines them. It is easy to see that 
\[ \ro_{p,\overline{A}} = \overline{\ro_{p,A}}\]
and therefore $\ro_{p,A}$ and $\ro_{p,\overline{A}}$ are equivalent. 
The converse is provided by the following: 

\begin{proposition}\label{inequivalent}
Let $g\geq 3$, $A$ and $B$ be distinct primitive $2p$-th roots of unity. 
The the representations $\ro_{p,A}$ and $\ro_{p,B}$ are 
equivalent if and only if $A=\overline{B}$. 
\end{proposition}

\begin{remark}
The same proof actually shows that the projective 
representations $\rho_{p,A}$ and $\rho_{p,B}$ are 
equivalent if and only if $A=B$ or  $A=\overline{B}$.
\end{remark}

Let now $\Gamma\to H_i$, $i=1,\ldots,m$, 
be a collection of representations of the group $\Gamma$.  
The subgroup $H\subset \prod_{i=1}^mH_i$ is called 
$\Gamma$-diagonal, if there exists a partition $A_1,\ldots,A_s$ of 
$\{1,2,\ldots,m\}$ such that: 
\begin{enumerate}
\item All factors $H_i$, with $i\in A_t$, $1\leq t\leq s$ are 
equivalent as  representations of $\Gamma$.  Pick up some $i_t\in A_t$.  
Given some  intertwining isomorphisms 
 $L_{j,i_t}: H_j\to H_{i_t}$,  $j\in A_t\setminus\{i_t\}$, we set:  
\[H_{A_t}=\{(x,  
(L_{j,i_t}(x))_{j\in A_t\setminus\{i_t\}}), \;\;  
x\in H_{i_t}\},\]    
which is the graph of  the 
homomorphism $\oplus_{j\in A_t\setminus\{i_t\}} L_{j,i_t}$.
\item Then there exist intertwining isomorphisms  as above with the property 
that the group $H$ contains $\prod_{1\leq t\leq s} H_{A_t}$. 
In particular, if all representations $H_i$ of $\Gamma$ are 
pairwise inequivalent, then 
$H=\prod_{i=1}^mH_i$.
\end{enumerate}

We then have the following Hall lemma from \cite{Ku}: 

\begin{lemma}[Hall Lemma]
Let $\Gamma$ be a subgroup of  the product $\prod_{i=1}^mH_i$ of the 
adjoint simple (i.e. connected, without center and whose Lie algebra 
is simple) Lie groups $H_i$.  
Assume that the projection of $\Gamma$ on each factor $H_i$ 
is Zariski dense. Then the Zariski closure of $\Gamma$ 
in  $\prod_{i=1}^mH_i$ is a $\Gamma$-diagonal subgroup.  
\end{lemma}

Therefore, the Hall lemma above shows that the Zariski closure of 
$\rho_p(M_g)$ is all of $\mathbb P\mathbb G_p(\R)$.
Now using (\cite{Ku}, Lemma 3.6) we obtain that 
$\ro_p(\widetilde{M_g})$ is Zariski dense in $\mathbb G_p(\R)$.  
This proves the theorem.

\section{Proof of Proposition \ref{weakZD}}

Recall that the (reduced) 
Burau representation $\beta:B_n\to GL(n-1,\Z[q,q^{-1}])$ 
is defined on the standard generators $g_1,g_2,\ldots,g_{n-1}$ 
of the braid group $B_n$ on $n\geq 3$ strands by the matrices: 
\[ \beta_q(g_1)=\left(\begin{array}{cc}
-q & 1 \\
0  & 1 \\
\end{array}
\right) \oplus {\mathbf 1}_{n-3},\]
\[ \beta_q(g_j)={\mathbf 1}_{j-2}\oplus 
\left(\begin{array}{ccc}
1 & 0 & 0 \\
q & -q & 1 \\
0 & 0  & 1 \\
\end{array}
\right) \oplus {\mathbf 1}_{n-j-2}, \:\: {\rm for} \:\: 2\leq j\leq n-2,\]
\[ \beta_q(g_{n-1})={\mathbf 1}_{n-3}\oplus 
\left(\begin{array}{cc}
1 & 0 \\
q & -q \\
\end{array}
\right). 
\]

Denote  by $q_p(A)$, where $A$ is a primitive $2p$-th root of unity, 
the following root of unity: 
\[ q_p(A)=\left\{\begin{array}{ll}
A^{-4}, & {\rm if}\: p=5 \;{\rm or }\;p\equiv 0({\rm mod}\: 2); \\
A^{-8}, & {\rm if}\: p\equiv 1({\rm mod}\: 2), p\geq 7.\\
\end{array}\right. \]

Squier proved that, after a suitable rescaling, 
$\beta_{q_p(A)}$ preserves a non-degenerate 
Hermitian form and hence it takes 
values in either $U(3)$ or $U(2,1)$, depending 
on the signature of the invariant form.

Now, whenever $p\geq 5$ is prime  the group $\beta_{q_p(A)}(PB_3)$ 
is neither finite nor abelian, so that it is dense in 
$SU(2)\subset SU(3)$ and hence $\beta_{q_p(A)}(PB_4)$ is dense in 
$SU(3)$.

\begin{lemma}
If  the signature of the Hermitian invariant form associated to $\beta_q$ 
is $(2,1)$ or $(1,2)$, then the image of $\beta_q$ in $SU(2,1)$ is dense. 
\end{lemma}
\begin{proof}
There are several proofs of this statement. It is a consequence of the 
results of Deligne and Mostow concerning hyper-geometric integrals 
(see \cite{Mo}) and the identification of the Burau representation 
as such one.  
Alternatively, we might use a result of Freedman, Larsen and Wang 
(see \cite{FLW}) subsequently reproved and extended by Kuperberg in 
(\cite{Ku}, Thm.1) saying that the Burau representation of $B_4$  
at a $2p$-th root of unity is Zariski dense in the group $SU(2,1)$. 
\end{proof}

We will relate now the quantum representation $\rho_{p,A}$ 
and the Burau representation $\beta_{q_p(A)}$. 
Specifically we embed $\Sigma_{0,5}$ into $\Sigma_g$ by means of curves 
$c_1,c_2,c_3,c_4,c_5$  as in the figure below. 

\begin{center}
\includegraphics[scale=0.4]{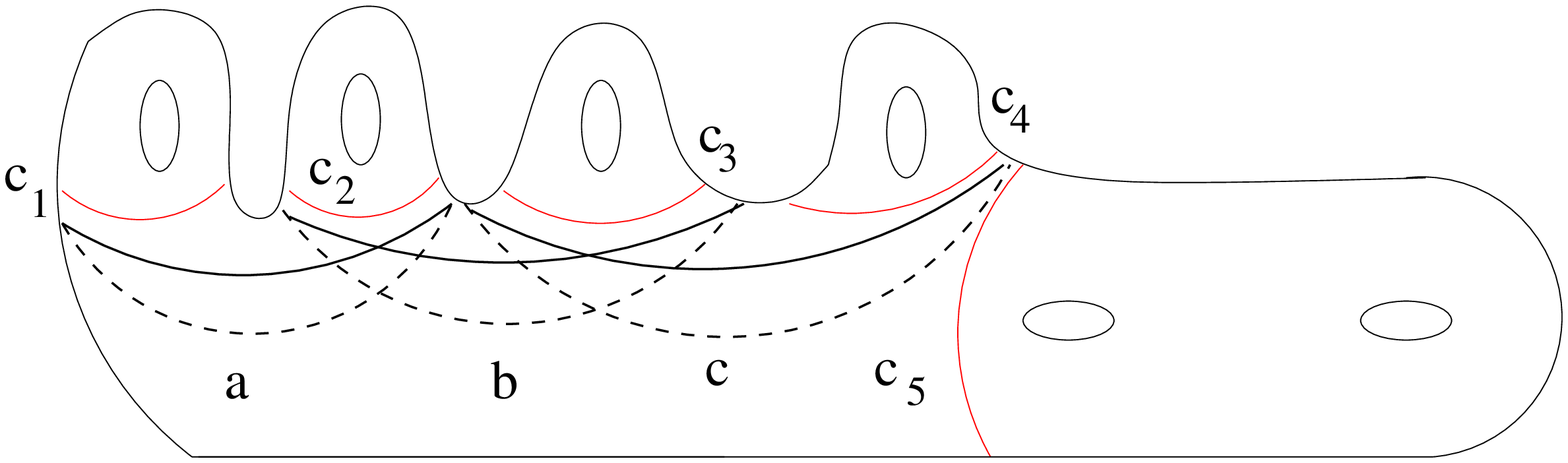}
\end{center}

\vspace{0.2cm}\noindent
Then the curves $a$, $b$ and $c$ 
which are surrounding two of the holes of $\Sigma_{0,5}$
The pure braid group $PB_4$ embeds into $M_{0,5}$ using a 
non-canonical splitting of the surjection $M_{0,5}\to PB_4$.  
Furthermore, $M_{0,5}$ embeds into $M_g$ when $g\geq 5$, by using the 
homomorphism induced by the inclusion of 
$\Sigma_{0,5}$ into $\Sigma_g$ as in the figure. 
When $g\in\{3,4\}$ we shall use other embeddings, similar to 
those used in \cite{FK11} for the proof that the image 
of the quantum representation contains free 
non-abelian subgroups.

\begin{lemma}\label{contain}
Let $p\geq 5$. The restriction of the quantum representation $\rho_p$ at 
$PB_4\subset M_{0,5}$ has an invariant 3-dimensional 
subspace such that the corresponding  sub-representation 
is equivalent to the Burau representation 
$\beta_{q_p(A)}$. 
\end{lemma}
\begin{proof}
For even $p$ and $PB_3$ instead of $PB_4$ this is the content 
of \cite{F}, Prop. 3.2. The odd case is similar. 
The invariant subspace is the space of conformal blocks 
associated to the surface $\Sigma_{0,5}$ with boundary labels  
$(2,2,2,2,2)$, when $p=5$ and $(4,2,2,2,2)$, when $p\geq 7$ 
respectively.  The eigenvalues of the half-twist  
can be computed as in \cite{F}. 
\end{proof}

Thus the image $\rho_p(PB_4)$ of the quantum representation 
projects onto the image of the Burau representation 
$\beta_{q_p(A)}(PB_4)$.

\vspace{0.2cm}\noindent
{\em End of the proof of Proposition \ref{weakZD}.} If $p$ is as above, then Proposition \ref{LW} 
shows that $\rho_{p,A_p}(M_g)$ is topologically dense in $SU(N(p,g))$ 
which is a compact Lie group. In particular, the image is 
Zariski dense in $SU(N(p,g))$.  The Zariski density is preserved 
by a Galois conjugacy and thus, $\rho_{p,A}(M_g)$ is Zariski dense 
in $SU(N_+(p,g,A), N_-(p,g,A))$ 
for every $A$.  Furthermore,  a Zariski dense subgroup of a reductive  
almost simple Lie group (in particular, of $SU(N_+(p,g,A), N_-(p,g,A))$) 
not contained in the center is either dense or discrete. 
Now, recall that $\rho_{p,A}(M_g)$ contains $\beta_{q_p}(PB_4)$. 
The latter group is topologically dense in $SU(3)$ and respectively 
$SU(2,1)$. Therefore  $\rho_{p,A}(M_g)$ is in-discrete. Since the 
image group $\rho_{p,A}(M_g)$ is not contained within the 
center of $SU(N_+(p,g,A), N_-(p,g,A))$, it follows that 
it should be topologically dense in $SU(N_+(p,g,A), N_-(p,g,A))$. 
This proves the claim.

\section{Proof of Proposition \ref{inequivalent}}
Assume that the representations $\ro_{p,A}$ and $\ro_{p,B}$ 
were equivalent as representations into the special pseudo-unitary group  
$SU(N_+(p,g,A),N_-(p,g,A))$. 
Then $\rho_{p,A}$ and $\rho_{p,B}$ are equivalent as 
projective representations  
into $PU(N_+(p,g,A),N_-(p,g,A))$ and hence by results of Walter (see 
\cite{Wal})  they are equivalent as representations into  
$U(N_+(p,g,A),N_-(p,g,A))$. 
This means that there exists an isomorphism 
$L:U(N_+(p,g,A),N_-(p,g,A))\to  U(N_+(p,g,B),N_-(p,g,B))$ with the 
property that 
\[ L\circ \ro_{p,A} =\ro_{p,B}\]
According to classical results of Rickart (improving previous 
results by Dieudonn\'e) from \cite{Ri} any automorphism $L$  
of a unitary group $\mathbb U$ over an infinite field could be expressed 
under the form: 
\[ L(x)= \chi(x) V x V^{-1} \]
where $V: \C^{N(p,g)}\to \C^{N(p,g)}$ is a semi-linear isomorphism, which 
will be called the intertwiner. 
Namely, this means that there exists some 
field automorphism $\phi:\C\to \C$ with the properties: 
\[\overline{\phi(a)}=\phi(\overline{a})\]
\[ V(a x+ b y)=\phi(a)V(x)+\phi(b)V(y), {\rm for }\; a,b\in \C, x,y\in \C^{N(p,g)}\]
\[ H(V(x),V(y))=\phi(V(x,y))\] 
where $\overline{a}$ denotes the conjugate of $a$ and 
$H$ is the Hermitian form associated to $A$ (or $B$). 
Furthermore, $\chi:\mathbb U \to U(1)$ is some homomorphism. 
It is well-known that the only automorphisms of $\C$ commuting with the 
complex conjugacy are the identity and the complex conjugacy. 
Thus $V$ is either a linear or an anti-linear isomorphism preserving 
(and respectively conjugating) the Hermitian form $H$.

Suppose therefore that $\ro_{p,A}$ and $\ro_{p,B}$ are equivalent. 
Let $V$ be the intertwiner and $\chi$ the character.

\begin{lemma}
If $g\geq 3$, then one can assume that the character $\chi$ is trivial.  
\end{lemma}
\begin{proof}
Let $\chi_0:\widetilde{M_g}\to U(1)$ denote the homomorphism defined by 
$\chi_0(x)=\chi(\ro_{p,A}(x))$. 
Since $\ro_{p,B}(x)=\chi(\ro_{p,A}(x)) V\circ \ro_{p,A}(x)\circ V^{-1}$,
we obtain that $\chi_0$ is a character on $\widetilde{M_g}$. 
The presentation given by Gervais (see \cite{Ger}, Thm.B) 
shows that lantern relations lift to $\widetilde{M_g}$. Furthermore, 
$\widetilde{M_g}$ is generated by lifts of pairwise conjugate 
Dehn twists along non-separating curves and one central generator $c$. 
We can express $c^{12}$ using the chain relation on 2-holed subtori with 
non-separating boundary curves. The lantern relation 
implies then that $H_1(\widetilde{M_g})=\Z/12\Z$. Alternatively, we can 
use the fact that the universal central extension  
$\widetilde{M_g}^{\rm univ}$ is a normal subgroup of index 12 
of $\widetilde{M_g}$ and $H_1(\widetilde{M_g}^{\rm univ})=0$.  
Therefore $\chi_0(x)^{12}=1$, for any $x\in \widetilde{M_g}$. 
On the other hand, it is 
known that $\ro_{p,A}(c)$ is the scalar matrix $A^{-12}$ (see e.g.
\cite{BHMV,MR,FP}) and hence $\chi_0(c)^p=1$.  Since $p\geq 5$ is 
prime, it follows that $\chi_0(c)=1$ and hence $\chi_0$ factors through 
$M_g$. Now, $M_g$ is perfect for $g\geq 3$ and thus the character 
$\chi_0$ is trivial. In particular, we can take $\chi$ to be trivial. 
\end{proof}

For each simple loop $\gamma$ on the surface,  
there is defined in (\cite{MR}, 4.3) 
a canonical lift $\widetilde{T}_{\gamma}\in \widetilde{M_g}$ 
of the (right) Dehn twist $T_{\gamma}$ along $\gamma$.

Consider the collection of eigenvalues of Dehn twists along 
with their multiplicities:  
\[ E(A, \widetilde{T}_{\gamma})=\{\lambda \; ; \;\;  {\rm there \; exists \; } x\neq 0 \; {\rm such \; that} \;  
\ro_{p,A}(\widetilde{T}_{\gamma})(x)=\lambda x\} \]

\begin{lemma}
If $\ro_{p,A}$ and $\ro_{p,B}$ are equivalent, then 
either $E(A,\widetilde{T}_{\gamma})= E(B, \widetilde{T}_{\gamma})$, for every $\gamma$ 
or else $E(A, \widetilde{T}_{\gamma})=\overline{E(B, \widetilde{T}_{\gamma})}$, 
for every $\gamma$. 
\end{lemma}
\begin{proof}
The linear representations $\ro_{p,A}$ and $\ro_{p,B}$ are 
either linearly or anti-linearly equivalent by means of some 
intertwiner isomorphism $V$. Now, if  
\[ \ro_{p,A}(\widetilde{T}_{\gamma})(x)= \lambda x \]
then we have: 
\[ \ro_{p,B}(\widetilde{T}_{\gamma})V(x)= V\ro_{p,A}(\widetilde{T}_{\gamma})V^{-1}(V(x))= \phi(\lambda) V(x) \]
Thus the set of eigenvalues of $\ro_{p,A}(\widetilde{T}_{\gamma})$ and 
$\ro_{p,B}(\widetilde{T}_{\gamma})$ should correspond to each 
another by means of $V$.
\end{proof}

Let $C=\{\gamma_1,\gamma_2,\ldots,\gamma_{n}\}$, $n\leq 3g-3$, be 
some set of disjoint simple closed curves 
on the closed surface of genus $g$. 
Then the  $n$ lifts of Dehn twists $\widetilde{T}_{\gamma_i}$, $\gamma_i\in C$ 
pairwise commute and hence one could diagonalize them 
simultaneously. Let $W_{p,g}$ denote the space of conformal 
blocks associated by the TQFT $\mathcal V_p$ to the closed 
orientable surface of genus $g$.  
We set $\Lambda=(\lambda_1,\lambda_2,\ldots,\lambda_{n})$ 
and define: 
\[ W (\Lambda,C, A)= \{x\in W_{p,g} \; ; \: \ro_{p,A}(\widetilde{T}_{\gamma_i})(x)=\lambda_i x, \; 
i\in \{1,2,\ldots,n\} \}\subset W_{p,g}\]

Let now $\Sigma_{g',n}\hookrightarrow \Sigma_g$ be an essential 
embedding of the surface $\Sigma_{g',n}$ of genus $g'$ 
with $n$ boundary components in $\Sigma_g$, i.e. 
such that the homomorphism induced 
at the level of fundamental groups is injective. 
We also assume that $\Sigma_{g',n}$ has a pants decomposition.  
Let $C$ be a maximal system 
of non null-homotopic and  pairwise non-homotopic  simple closed 
curves in $\Sigma_g\setminus \Sigma_{g',n}$.  
Then $C$ contains  the set $C_0$ of boundary circles of $\Sigma_{g',n}$. 
We associate to each circle $c_i$ from $C_0$ some  arbitrary $\lambda_i$. 
Further, to each circle  $c_s$ in $C-C_0$ we associate  $\lambda_s=1$.   
We say that $\Lambda$ comes from a coloring of $C_0$ if 
there is some set of colors  $(i_1,\ldots,i_n)$ such that, for all  
 \[ \lambda_j = \left\{\begin{array}{ll}
A^{i_j(i_j+2)}, &  {\rm for }\; {\rm odd}\; p; \\
(-1)^{i_j} A^{i_j(i_j+2)}, & {\rm for }\; {\rm even}\; p.\\
\end{array}\right.  \] 

\begin{lemma}
The vector space  $W(\Lambda, C, A)$ is non-zero only if 
$\Lambda$ comes from a coloring of $C_0$. In this case 
$W(\Lambda, C, A)$ can be identified with the 
space of conformal blocks associated to the subsurface $\Sigma_{g',n}$ 
and the coloring $(i_1,\ldots,i_n)$ of the boundary components. 
\end{lemma}
\begin{proof}
It is well-known that the eigenvalues $\lambda_{i}$ of Dehn twists 
in the basis given by colored trivalent graphs from (\cite{BHMV}, 4.11) 
are given by the formula  from above, in terms 
of colors (see also \cite{BHMV}, 5.8).  
Now  the space of conformal blocks associated to a subsurface 
splits  as a direct sum of one dimensional eigenspaces 
$W(\Lambda,C\cup C',A)$ associated to 
all possible colorings of some maximal system $C'$  
of non null-homotopic and  pairwise non-homotopic  simple closed 
curves on the subsurface 
$\Sigma_{g',n}$.
\end{proof}

\begin{lemma}
Any intertwiner $V$ 
induces an isomorphism  
\[ V: W(\Lambda,C, A) \to  W(\Lambda,C,B),\]
and hence an isomorphism between the subspaces of 
conformal blocks associated to any 
essential subsurface with colored boundary. 
\end{lemma}
\begin{proof}
This is a consequence of the previous two lemmas.
\end{proof}

\begin{lemma}
Let $A$ and $B$ be primitive $2p$-th roots of unity. 
\begin{enumerate} 
\item If $\ro_{p,A}$ and $\ro_{p,B}$ are linearly equivalent, then $A=B$.  
\item If $\ro_{p,A}$ and $\ro_{p,B}$ are anti-linearly equivalent, then   
$A=\overline{B}$. 
\end{enumerate}
\end{lemma}
\begin{proof}
Consider first the case (of the $SO(3)$-TQFT) when 
$p=2r+1$ is  odd. Then the set of colors 
is $\mathcal C=\{0,2,4,\ldots, 2r-2\}$ and the collection of  eigenvalues 
is (see \cite{BHMV}, 5.8) given by 
$E(A, \widetilde{T}_{\gamma})=(A^{i(i+2)})_{i\in \mathcal C}$. 
By assumption there exists a bijection  between the set of colors 
associated to $A$ and $B$, namely $f:\mathcal C\to \mathcal C$ 
with the property that 
$A^{i(i+2)}=B^{f(i)(f(i)+2)}$. 
Moreover, the bijection $f$ between the colors should be compatible with 
the conformal blocks structure.

We will prove by recurrence on $i$  that $f(i)=i$. 
First we have $f(0)=0$. 
If $f(2)\neq 2$, then $f(2) > 2$. The space of conformal blocks 
associated to a 4-holed sphere $\Sigma_{0,4}$ whose boundary 
components are colored by $2$ has dimension $3$, corresponding 
to the colors $\{0,2,4\}$ on a separating curve (if $p\geq 7$). 
The bijection $f$ should send this space into the space of  
conformal blocks associated to a 4-holed sphere whose boundary components 
are colored by $f(2)$. But the separating curve can be given any 
even color in the range $0,2, \ldots, f(2)$. Thus $f(2) >2$ would 
lead to a contradiction. This implies that $f(2)=2$.  
Therefore we have $A^8=B^8$ where both $A$ and $B$ are primitive 
roots of unity of order $2p$, with odd $p$, 
which implies that $A=B$.   
The case $p=5$ is immediate, by direct calculation.

Consider now the case (of the $SU(2)$-TQFT) when 
$p=2r+2$ is  even. Then the set of colors 
is $\mathcal D=\{0,1,2,\ldots, r-1\}$ and the collection of  eigenvalues 
is (see \cite{BHMV}, 5.8) given by 
$E(A, T_{\gamma})=((-1)^iA^{i(i+2)})_{i\in \mathcal D}$. 
By assumption there exists a bijection  between the set of colors 
associated to $A$ and $B$, namely $f:\mathcal D\to \mathcal D$ 
with the property that 
$(-1)^iA^{i(i+2)}=(-1)^{f(i)}B^{f(i)(f(i)+2)}$. 
Moreover, the bijection $f$ between the colors should be compatible with 
the conformal blocks structures.

We first have $f(0)=0$. 
If $f(1)\neq 1$, then $f(1) > 1$. The space of conformal blocks 
associated to a 4-holed sphere $\Sigma_{0,4}$ whose boundary 
components are colored by $1$ has dimension $2$, corresponding 
to the colors $\{0,2\}$ on a separating curve (if $p\geq 4$). 
The space of conformal blocks associated to 
a 4-holed sphere $\Sigma_{0,4}$ whose boundary components are colored 
by $f(1)$ has dimension $f(1)+1$, corresponding 
to the colors $\{0,2,\ldots, \min(2f(1), 2r-2-2f(1))\}$ on a separating curve. 
This leads to a contradiction when $f(1) >1$ and $r\geq 3$. Thus $f(1)=1$. 
Moreover, the space of conformal blocks associated 
to the 3-holed sphere with boundary components colored by $1, 1$ and $2$ 
is 1-dimensional so this is the same for the 
the coloring  $1,1$ and $f(2)$. Therefore  $f(2)\leq 2$ by the 
Clebsch-Gordan admissibility conditions and hence $f(2)=2$.  
Therefore $A^3=B^3$ and $A^8=B^8$ which gives us $A=B$.

The above proof works without essential modifications when 
$V$ is anti-linear. 
\end{proof}

\subsection{The first Johnson subgroups and their  quantum images}
For a group $G$ we denote by $G_{(k)}$ the lower central series 
defined by:  
\[ G_{(1)}=G, \, G_{(k+1)}=[G,G_{(k)}], k\geq 1\]
An  interesting  family of subgroups of the mapping class group is 
the set of higher Johnson subgroups defined as follows. 

\begin{definition}
The $k$-th Johnson subgroup 
$I_g(k)$ is the group of mapping classes of homeomorphisms of the closed 
orientable surface $\Sigma_g$ of genus $g$ 
whose action by outer automorphisms on $\pi/\pi_{(k+1)}$ is trivial, where 
$\pi=\pi_1(\Sigma_g)$. 
\end{definition}

Thus $I_g(0)=M_g$, $I_g(1)$ is the Torelli group commonly 
denoted $T_g$, while $I_g(2)$ is the group generated by the Dehn twists along 
separating simple closed curves and considered by Johnson and Morita 
(see e.g. \cite{John,Mor}), which is often denoted by ${K}_g$.  

\begin{proposition}\label{chain3}
For $g\geq 3$ we have the following chain of normal groups of finite index:
\[ \rho_p([K_g,K_g])\subset \rho_p(I_g(3))\subset \rho_p(K_g)\subset \rho_p(T_g)\subset \rho_p(M_g).\] 
\end{proposition}
\begin{proof}
There is a surjective homomorphism 
$f_1:Sp(2g,\Z)\to \frac{\rho_p(M_g)}{\rho_p(T_g)}$. 
The image of the $p$-th power  of a Dehn twist in $\rho_p(M_g)$ is trivial. 
On the other hand,  the image of a Dehn twist in $Sp(2g,\Z)$ is a 
transvection and taking all Dehn twists one obtains a system of generators  
for $Sp(2g,\Z)$. Using the congruence subgroup property 
for $Sp(2g,\Z)$, where $g\geq 2$, the image of $p$-th powers 
of Dehn twists in $Sp(2g,\Z)$ generate the congruence subgroup 
$Sp(2g,\Z)[p]=\ker (Sp(2g,\Z)\to Sp(2g,\Z/p\Z))$. 
Since the mapping class group is generated by Dehn twists 
the homomorphism $f_1$ should factor through 
$\frac{Sp(2g,\Z)}{Sp(2g,\Z)[p]}=Sp(2g,\Z/p\Z)$. In particular,  
the image of $f_1$ is finite.

By the work of Johnson (see \cite{John}) one knows that 
when $g\geq 3$ the quotient $\frac{T_g}{K_g}$ is a finitely generated 
abelian group $A$ 
isomorphic to $\bigwedge^3H/H$,  where 
$H$ is the homology of the surface. Thus there is a surjective 
homomorphism 
$f_2: A\to \frac{\rho_p(T_g)}{\rho_p(K_g)}$.  
The Torelli group $T_g$ is generated by  
BP-pairs, namely elements of the form 
$T_{\gamma}T_{\delta}^{-1}$, where $\gamma$ and $\delta$ are 
non-separating disjoint simple closed curves bounding a subsurface 
of genus 1 (see \cite{John79}). Now $p$-th powers of the BP-pairs 
$(T_{\gamma}T_{\delta}^{-1})^p= T_{\gamma}^pT_{\delta}^{-p}$ 
have trivial images  in $\frac{\rho_p(T_g)}{\rho_p(K_g)}$. 
But the classes 
$T_{\gamma}T_{\delta}^{-1}$ also generate 
the quotient $A$ and hence the classes 
$(T_{\gamma}T_{\delta}^{-1})^p$ will generate the  abelian subgroup 
$pA$ of  those elements of $A$ which are divisible by $p$. 
This shows that $f_2$ factors through 
$A/pA$, which is a finite group because $A$ is finitely generated. 
In particular, $\frac{\rho_p(T_g)}{\rho_p(K_g)}$ 
is finite.

Eventually, $\frac{K_g}{I_g(3)}$ is also a finitely generated abelian group 
$A_3$, namely the image of the third  Johnson homomorphism. 
Since $K_g$ is generated by the Dehn twists along separating 
simple closed curves the previous argument shows that 
$\frac{\rho_p(K_g)}{\rho_p(I_g(3))}$ is  
the image of a surjective homomorphism from $A_3/pA_3$ and hence is 
finite.

Recently, Dimca and Papadima proved in \cite{DP} 
that $H_1(K_g)$ is finitely generated for $g\geq 3$. 
The above proof  implies 
that $\rho_p([K_g,K_g])$ is of finite index in $\rho_p(K_g)$.
We have also the following alternative 
argument, which makes the proof independent of the result in \cite{DP}. 
The group $\rho_p(K_g)$ is finitely generated since
it is of finite index in the finitely generated group $\rho_p(M_g)$.
Thus $\rho_p(K_g)/\rho_p([K_g,K_g])$ is abelian and finitely generated. 
Moreover $\rho_p(K_g)/\rho_p([K_g,K_g])$ is generated by torsion elements of 
order $p$ , since $K_g$ is generated by Dehn twists along bounding simple 
closed curves. Now, any finitely generated abelian group generated 
by order $p$ elements must be finite. 
\end{proof}

\begin{remark}
A natural question is whether $\rho_p(I_g(k+1))$ is of finite index 
in $\rho_p(I_g(k))$, for every $k$. The arguments above break down 
at $k=3$ since there are no products of powers 
of commuting Dehn twists in any higher Johnson subgroups. 
More specifically, we have to know the image of  the group 
$M_g[p]\cap I_g(k)$ in  $\frac{I_g(k)}{I_g(k+1)}$
by the Johnson homomorphism. Here $M_g[p]$ denotes the 
normal subgroup generated 
by the $p$-th powers of Dehn twists. If the image were a lattice in 
$\frac{I_g(k)}{I_g(k+1)}$, then we could deduce as above that 
$\frac{\rho_p(I_g(k))}{\rho_p(I_g(k+1))}$ is finite. 
\end{remark}

{
\small      
      
\bibliographystyle{plain}

}

\end{document}